\documentclass[12pt,letterpaper]{amsart}
\usepackage{url}
\usepackage{fullpage}
\usepackage{pslatex} 
\usepackage{amsmath}
\usepackage{amssymb}
\usepackage{amsthm}
\usepackage{color}
\usepackage{array}
\usepackage{xy}
\usepackage{setspace}
\usepackage{multicol}
\usepackage{hyperref}
\usepackage{cite}
\usepackage[margin=1in,letterpaper]{geometry}
\newcommand{\si}{\mathbf{S}}

\newtheoremstyle{slanted}
{3pt}
{3pt}
{\slshape}
{}
{\bfseries}
{.}
{.5em}
{}
\theoremstyle{slanted}
\newtheorem{thm}{Theorem}[section]
\newtheorem{defn}[thm]{Definition}
\newtheorem{lem}[thm]{Lemma}
\newtheorem{prop}[thm]{Proposition}

\newtheorem{cor}[thm]{Corollary}

\theoremstyle{remark}
\newtheorem{rem}[thm]{Remark}

\usepackage{fancyhdr}
\begin{document} 

\title[A New Approach to Enumerating Statistics Modulo $n$]{A New Approach to Enumerating Statistics Modulo $n$}
\author{William Kuszmaul}
\maketitle
\begin{abstract}
We find a new approach to computing the remainder of a polynomial modulo $x^n-1$; such a computation is called modular enumeration. Given a polynomial with coefficients from a commutative $\mathbb{Q}$-algebra, our first main result constructs the remainder simply from the coefficients of residues of the polynomial modulo $\Phi_d(x)$ for each $d\mid n$. Since such residues can often be found to have nice values, this simplifies a number of modular enumeration problems; indeed in some cases, such residues are already known while the related modular enumeration problem has remained unsolved. We list six such cases which our technique makes easy to solve. Our second main result is a formula for the unique polynomial $a$ such that $a \equiv f \mod \Phi_n(x)$ and $a\equiv 0 \mod x^d-1$ for each proper divisor $d$ of $n$.

We find a formula for remainders of $q$-multinomial coefficients and for remainders of $q$-Catalan numbers modulo $q^n-1$, reducing each problem to a finite number of cases for any fixed $n$. In the prior case, we solve an open problem posed by Hartke and Radcliffe. In considering $q$-Catalan numbers modulo $q^n-1$, we discover a cyclic group operation on certain lattice paths which behaves predictably with regard to major index. We also make progress on a problem in modular enumeration on subset sums posed by Kitchloo and Pachter.
\end{abstract}

\section{Introduction}

In this paper, we provide a novel approach to modular enumeration, allowing us to solve problems that were previously unapproachable. Modular enumerations appear in widely ranging fields of mathematics, from representation theory to single-error correcting codes. Interesting applications appear, for example, in  \cite{Reiner04,Kitchloo93,Stanley72,Hanlon98,Desarmenien90,Od78,Li12,Cohen55,Nicol54,Brunat12,Desarmenien94}. 

Let $f$ be a polynomial with coefficients from a commutative $\mathbb{Q}$-algebra. Given such an $f$, modular enumeration is defined as finding the remainder $\si_n(f)$ of $f$ modulo $x^n-1$ for a positive integer $n$. (We call $\si_n(f)$ the $n$-simplification of $f$.) Equivalently, $$\si_n(f)=\sum_{0\le i<n}\ \ \sum_{j\equiv i \mod n}[x^j]f\cdot x^i.$$

Let $\Phi_d(x)$ be the $d$-th cyclotomic polynomial, and choose a polynomial $m_d$ satisfying $m_d \equiv f \mod \Phi_d(x)$ for each $d\mid n$. (In various practical cases, these $m_d$ can be chosen to be much simpler than $f$.) Our main result, Theorem \ref{thmmodcount}, constructs the coefficients of $\si_n(f)$ in terms of the coefficients of $m_d$ for $d\mid n$ and in terms of Ramanujan sums (which have a short closed-form expression). For some generating functions $f$, results about $m_d$ are already known, while results about the $n$-simplification often are not. Examples include $q$-multinomial coefficients \cite{Sagan92}, \textbf{INV}$\mathbf{_n}$ of alternating permutations \cite{Desarmenien83}, $(q, t)$-Eulerian polynomials \cite{Desarmenien92}, modified Hall-Littlewood polynomials \cite{Lascoux94}, $(q,t)$-Kostka polynomials \cite{Descouens06}, Desarmenien's $C_{m,n}(q)$ \cite{Desarmenien83}, and Desarmenien's $\mathbf{T_{a}}$ \cite{Desarmenien83}; in each of these cases, Theorem \ref{thmmodcount} can be used to effortlessly obtain new  modular enumeration results.

That such an approach to modular enumeration exists should not be entirely surprising, given that $\prod\limits_{d\mid n} \Phi_d(x)=x^n-1$ and that all $\Phi_d(x)$ are pairwise coprime. But the Chinese Remainder Theorem, while postulating the existence of a construction of $\si_n(f)$ from given remainders modulo $x^n-1$, does not yield any manageable formulas for its coefficients. In order to derive such a formula, we introduce an invariant modulo $\Phi_n(x)$ which is interesting in its own right. After providing background in Section \ref{Background}, in Section \ref{Secmain}, we prove Theorem \ref{thmmodcount} using this invariant. This leads us to our second main result, a formula using the coefficients of $f$ which outputs the unique polynomial $a\equiv f \mod \Phi_n(x)$ such that $a \equiv 0 \mod x^d-1$ for each proper divisor $d$ of $n$ (Theorem \ref{thmrepelement}).

In Section \ref{Secap}, we show two example applications of Theorem \ref{thmmodcount}. In Subsection \ref{Subbin}, we find a formula for the $n$-simplification of the $q$-multinomial coefficient ${ j \brack k_1, k_2, \ldots, k_l }$ which reduces the problem to the cases where $j<n$, resolving an open problem posed in \cite{Hartke13}. Furthermore, in the case of $n|j$, our formula can be reduced to be non-recursive. This is a significant improvement on a recursive formula found in \cite{Hartke13} for the same case. Given that the coefficients of ${ j \brack k_1,k_2,\ldots, k_l}$ have no simple closed-form expression, it seems unlikely that the remaining cases, where $j<n$, can have one either. In Subsection \ref{Subsum}, we make progress on an open problem on subset sums modulo $n$ previously posed in \cite{Kitchloo93} and solved in some cases in \cite{Stanley72,Kitchloo93,Li12}. 

The $q$-Catalan number\footnote{Also known as MacMahon's maj-statistic $q$-Catalan number.} $C_j(q)$ is a natural $q$-analogue for the Catalan numbers. Perhaps the most intuitive definition of the polynomial $C_j(q)$ is as the generating function for major index of Dyck words containing precisely $j$ zeros and $j$ ones \cite{Furlinger85}. In Section \ref{Seccat}, we find a formula for the $n$-simplification of $C_j(q)$, reducing the problem to cases where $j<n$. Once again, it seems unlikely that the remaining cases have a manageable formula, given that the coefficients of $C_j(q)$ do not and that $\si_n(C_j(q))=C_j(q)$ for $n$ sufficiently large relative to $j$. We provide two proofs of our formula. The first uses a previously undiscovered group operation on certain lattice paths which allows for us to find simple values of $m_d$ and then obtain $\si_n(C_j(q))$ with Theorem \ref{thmmodcount}. This group action is interesting in its own right, cyclically shifting major index modulo $n$. 

Finally, we conclude with discussion and open questions in Section \ref{Conclusion}.

\section{Background and Conventions}\label{Background}

For the entirety of this paper, we use $\mathbb{K}$ to denote a commutative $\mathbb{Q}$-algebra. When $a$, $b$ and $c$ are elements of a commutative ring $R$, we will use $a\equiv b\mod c$  as a shorthand for $a\equiv b \mod cR$, at least when $R$ can be uniquely inferred from the context.

\begin{defn}
An \emph{$n$-simplified polynomial} is a polynomial of degree less than $n$. The \\\emph{$n$-simplification} $\si_n(f)$ of a polynomial $f \in \mathbb{K}[x]$ is the $n$-simplified polynomial $h \in \mathbb{K}[x]$ satisfying $h \equiv f \mod x^n-1$ (that is, the remainder of $f$ modulo $x^n-1$).
\end{defn}

We follow the convention that whenever $f$ is a polynomial, $[x^i]f$ is the coefficient of $x^i$ in $f$.

\begin{defn}
Given any $f \in \mathbb{K}[x]$, $i \in \mathbb{Z}$, and $n\in \mathbb{N}$, we may refer to $\sum\limits_{j \equiv i \mod n}[x^j]f$ as $\si^i_n(f)$.
\end{defn}

Given $i \in \mathbb{Z}$ with $0\le i<n$ and $f\in \mathbb{K}[x]$, we have $\si^i_n(f)=[x^i]{\si_n(f)}$.

\begin{defn}
The \emph{cyclotomic polynomials} $(\Phi_n(x))_{n\geq 1}$ are a sequence of polynomials in $\mathbb{Z}[x]$ defined recursively by the equality
\[
\prod_{d\mid n} \Phi_d(x) = x^n-1\qquad \text{for all }n\geq 1.
\]
\end{defn}

While this is not immediately clear, these polynomials are actually well-defined. The polynomial $\Phi_n(x)$ is called the \emph{$n$-th cyclotomic polynomial}, and its roots (in the algebraic closure of $\mathbb{Q}$) are the \emph{primitive $n$-th roots of unity}. One can show that $\Phi_n(x)$ is monic for every $n$. Moreover, $\Phi_n(x)$ is known to be irreducible in $\mathbb{Q}[x]$ (a fact we will not actually end up using).

As a convention, we use $(j,k)$ for $j,k \in \mathbb{Z}$ to denote the GCD of $j$ and $k$. Furthermore, $\phi(n)$ is Euler's totient function of $n$.

\begin{defn}
For any positive integer $n$, we use $W_n$ to denote the set of $n$-th primitive roots of unity in a fixed algebraic closure of $\mathbb{Q}$. Note that $|W_n|=\phi(n)$.
\end{defn}

The following simple fact will be applied in the derivation of our results:

\begin{lem}\label{lempolfrac}
Let $d$ and $j$ be two divisors of a positive integer $n$.
\begin{itemize}
\item If $d \nmid j$, then $\frac{x^n-1}{x^j-1} \equiv 0 \mod \Phi_d(x)$ in $\mathbb{Z}[x]$. 
\item If $d \mid j$, then $\frac{x^n-1}{x^j-1} \equiv \frac{n}{j} \mod \Phi_d(x)$ in $\mathbb{Z}[x]$.
\end{itemize}
\end{lem}
\begin{proof}
Note that given $w \in W_d$, $\frac{w^n-1}{w^j-1}=1+w^j+\cdots+w^{n-j}$ is $0$ when $d \nmid j$ and is $\frac{n}{j}$ when $d \mid j$. Since the roots of $\Phi_d(x)$ are $W_d$ (in the algebraic closure of $\mathbb{Q}$), both congruences claimed in this lemma hold in $\mathbb{Q}[x]$. Therefore, noting that $\Phi_d$ is monic, both claims hold by Gauss's lemma.
\end{proof}


Although we will not use this, it has also been rather nicely shown in \cite{Bruijn53} that the ideal in $\mathbb{Z}[x]$ generated by $\Phi_n(x)$ is also generated by $\frac{x^n-1}{x^{n/p}-1}$ for all prime $p\mid n$.

\begin{defn}
For any integer $l$ and positive integer $n$, the \emph{Ramanujan sum} for $n$ and $l$ is defined as the integer $c_n(l) = \sum\limits_{d \mid n,l} \mu(\dfrac{n}{d}) d$. By abuse of notation, we will often regard this sum as an element of $\mathbb{K}$.
\end{defn}

The following characterization of Ramanujan sums is often used as an alternative definition: 

\begin{lem}\label{lemramanujanroots}
Every integer $l$ and positive integer $d$ satisfy $c_d(l) = \sum\limits_{w\in W_d}w^l$.
\end{lem}

\begin{proof}
For every integer $l$ and positive integer $d$, let $c'_d(l) = \sum\limits_{w\in W_d}w^l$. We need to prove that $c'_d(l) = c_d(l)$.
Since the $n$-th roots of unity are precisely the $d$-th primitive roots of unity for $d$ ranging over the divisors of $n$, we have $$\sum\limits_{d\mid n} c'_d(l)=\sum\limits_{w^n=1}w^l=\begin{cases}n \text{ if } n\mid l \\ 0 \text{ if } n\nmid l \end{cases}$$ for every positive integer $n$.
By M\"{o}bius inversion, this implies $c'_n(l)=\sum\limits_{d\mid n,l}\mu(\frac{n}{d})d = c_n(l).$
\end{proof}

Note that $c_n(1)=\mu(n)$ and $c_n(0)=\phi(n)$. More generally, the following closed-form expression for all Ramanujan sums is due to H\"older:
$$c_n(l)=\sum_{w \in W_n}w^l=\frac{\mu(\frac{n}{(n,l)})\phi(n)}{\phi(\frac{n}{(n,l)})}.$$
(This can easily be derived from the observations that $w^l \in W_{n/(n,l)}$ for $w \in W_n$, and that $w^l$ is uniformly distributed on $W_{n/(n,l)}$ as $w$ ranges over $W_n$.)

\begin{defn}
Let $S$ be a set and $\equiv$ be an equivalence relation on $S$. A function $f$ on $S$ is an \emph{invariant} (on $S$ with respect to $\equiv$) if for $w,w'\in S$, we have $w\equiv w' \implies f(w)=f(w')$. If, in addition, $f(w)=f(w')\implies w\equiv w'$, then $f$ is a \emph{complete invariant}  (on $S$ with respect to $\equiv$). If the equivalence is congruence modulo some ideal $I$, then we will refer to invariants (resp. complete invariants) with respect to this equivalence as ``invariants (resp. complete invariants) modulo $I$''.
\end{defn}

\begin{defn}
Let $n$ and $d$ be positive integers. A polynomial $a \in \mathbb{K}[x]$ is \emph{periodic on $n$} (with \emph{period $d$}) if $d \not\equiv 0 \mod n$ and $\si^i_n(a)={\si^{i+d}_n(a)}$ for all $i$.
\end{defn}
For example, $1+2x+x^2+2x^3+x^4+2x^5+x^6+2x^7$ is periodic on $8$ with periods $2$ and $4$. It is also periodic on $4$ with period $2$. But it is not periodic on $2$.

\begin{lem}\label{lemperiodzero}
Let $a \in \mathbb{Z}[x]$ be periodic on $n$. Then $a \equiv  0 \mod \Phi_n(x)$.
\end{lem}
\begin{proof}
We can assume $a$ is $n$-simplified because $x^n-1 \equiv 0 \mod \Phi_n(x)$ and because the \\$n$-simplification of a polynomial periodic on $n$ is still periodic on $n$. For an $n$-simplified polynomial to be periodic on $n$, it must be $r\frac{x^n-1}{x^d-1}$ for some $r \in \mathbb{Z}[x]$ and some proper divisor $d$ of $n$. If follows from Lemma \ref{lempolfrac} that $a \equiv 0 \mod \Phi_n(x)$.
\end{proof}

Although we will not be using this, it is straightforward to see that Lemma \ref{lemperiodzero} actually holds in $R[x]$ for any ring $R$. 

\section{An invariant modulo $\Phi_n(x)$} \label{Secmain}

In this section, we find a previously unknown invariant modulo $\Phi_n(x)$ which leads us to our two main results, Theorem \ref{thmmodcount} and Theorem \ref{thmrepelement}. In addition, we find a complete invariant on $\mathbb{Z}[x]$ modulo $\Phi_n(x)$. Along the way, we obtain a construction that, given a polynomial $m_d \in \mathbb{K}[x]$ for each $d\mid n$, finds the unique $n$-simplified polynomial in $\mathbb{K}[x]$ which is congruent to $m_d \mod \Phi_d(x)$ for each $d\mid n$.

\begin{defn}
Given $a \in \mathbb{K}[x]$,  $i \in \mathbb{Z}$, and a positive integer $n$, let $G^n_i(a)=\sum\limits_{s\ge 0}[x^s]a\cdot c_n(i-s)$.
\end{defn}

\begin{defn}
Given $a \in \mathbb{K}[x]$ and a positive integer $n$, we define $G^n(a)$ as $\frac{1}{n}\sum\limits_{0\le i <n}G^n_i(a)x^i$.
\end{defn}

\begin{lem}\label{thminvariant}
Let $n$ be a positive integer. Let $a,b \in \mathbb{K}[x]$ be such that $a \equiv b \mod \Phi_n(x)$. Then $G^n_i(a)=G^n_i(b)$ for every $i \in \mathbb{Z}$.
\end{lem}
\begin{proof}
Since $G^n_i: \mathbb{K}[x] \rightarrow \mathbb{K}$ is a linear map, we need only show that $G^n_i(x^t\Phi_n(x))=0$ for every $i \in \mathbb{Z}$ and non-negative $t\in \mathbb{Z}$. By virtue of the easy-to-check identity $G^n_i(x^t c) = G^n_{i-t}(c)$ for every $c \in \mathbb{K}[x]$, we can reduce this to the case $t=0$. In other words, we only have to check that $G^n_i(\Phi_n(x))=0$ for every $i \in \mathbb{Z}$. Expanding $G^n_i(\Phi_n(x))$ and $c_n(i-s)$ (the latter by way of Lemma \ref{lemramanujanroots}) gives $$G^n_i(\Phi_n(x))=\sum\limits_{s\ge 0}[x^s]\Phi_n(x)\cdot c_n(i-s)=\sum\limits_{s\ge 0}[x^s]\Phi_n(x)\sum\limits_{w\in W_n}w^{i-s}.$$ 
Rearranging the expression on the right yields
$$\sum\limits_{w \in W_n}w^i\sum\limits_{s \ge 0}[x^s]\Phi_n(x)w^{-s}=\sum\limits_{w\in W_n}w^{i}\Phi_n(w^{-1}).$$
For each $w\in W_n$, we have $w^{-1}\in W_n$, implying $\Phi_n(w^{-1})=0$. Thus $G^n_i(\Phi_n(x))=0$.
\end{proof}

\begin{thm}\label{thmmodcount}
Let $a \in \mathbb{K}[x]$, let $n$ be a positive integer, and let $i\in \mathbb{Z}$. For each $d\mid n$, let $m_d \in \mathbb{K}[x]$ be such that $m_d \equiv a \mod \Phi_d(x)$. Then
$$\si^i_n(a)=\frac{1}{n}\sum_{d\mid n}{G^d_i(m_d)}.$$
\end{thm}
\begin{proof}
Observe that $\frac{1}{n}\sum\limits_{d\mid n}{G^d_i(m_d)}=\frac{1}{n}\sum\limits_{d\mid n}{G^d_i(a)}$ (since respective addends of these sums are equal by Lemma \ref{thminvariant}). Expanding yields $$\frac{1}{n}\sum\limits_{d\mid n}\sum_{s\ge 0}[x^s]a\cdot c_d(i-s)=\frac{1}{n}\sum_{s\ge 0}[x^s]a\sum_{d\mid n}c_d(i-s).$$  From Lemma \ref{lemramanujanroots} and the observation that the $n$-th roots of unity are the primitive $d$-th roots of unity for $d$ ranging over the divisors of $n$, we note that
$$\sum\limits_{d\mid n}c_d(i-s)=\sum\limits_{w^n=1}w^{i-s}=\begin{cases}n \text{ if } n\mid i-s \\ 0 \text{ if } n\nmid i-s \end{cases}.$$ 
Plugging this in brings us to $\sum\limits_{s \equiv i \mod n}[x^s]a=\si^i_n(a)$.
\end{proof}

\begin{cor}\label{cormodcount}
$$\si^i_n(a)=\frac{1}{n} \sum\limits_{d|n} \sum\limits_{0 \leq s <d} \si^s_d(m_d) \cdot c_d(i-s).$$
\end{cor}
\begin{proof}
Note that $m_d \equiv \si_d(m_d) \mod \Phi_d(x)$. Applying Theorem \ref{thmmodcount}, we thus get
$$\si^i_n(a)=\frac{1}{n}\sum_{d\mid n}{G^d_i(\si_d(m_d))}=\frac{1}{n}\sum_{d \mid n}\sum_{s\ge 0}[x^s]\si_d(m_d) \cdot c_d(i-s)=\frac{1}{n}\sum_{d \mid n}\sum_{s\ge 0}\si^s_d(m_d) \cdot c_d(i-s).$$
\end{proof}


\begin{lem}\label{lembuildresidues}
Let $n$ be a positive integer. For each $d\mid n$, let $a_d \in \mathbb{K}[x]$. Then there exists a unique $n$-simplified $a \in \mathbb{K}[x]$ such that $a\equiv a_d \mod \Phi_d(x)$ for each $d\mid n$.
\end{lem}
\begin{proof}
We start with an auxilary construction. Let $c \in \mathbb{K}[x]$. Let $j$ be a divisor of $n$, and let $r \in \mathbb{K}[x]$. Let $$c'=c+\frac{(r-c)j}{n}(1+x^{j}+x^{2j}+\cdots+x^{n-j})=c+\frac{(r-c)j}{n}\cdot\frac{x^n-1}{x^j-1}.$$
For each $d\mid n$ with $d \nmid j$,  we have $c' \equiv c \mod \Phi_d(x)$ since $\frac{x^n-1}{x^j-1} \equiv 0 \mod \Phi_d(x)$ by Lemma \ref{lempolfrac}. Suppose instead that $d\mid j$. Then since $\frac{x^n-1}{x^j-1} \equiv \frac{n}{j} \mod \Phi_d(x)$ (by Lemma \ref{lempolfrac}), we have that $$c'=c+\frac{(r-c)j}{n}\cdot\frac{x^n-1}{x^j-1} \equiv c+\frac{(r-c)j}{n}\cdot\frac{n}{j}=r \mod \Phi_d(x).$$ Thus $c' \equiv r \mod \Phi_d(x)$ for each $d\mid j$ and $c' \equiv c \mod \Phi_d(x)$ for each $d\mid n$ with $d\nmid j$. The construction of $c'$ from $c$, $j$, and $r$ will be referred to as Construction (1). 

Let $\{d_1,d_2,d_3,\ldots,d_t\}$ be an ordered list of the divisors of $n$ such that $d_1>\cdots > d_t$. Notice that $d_l \nmid d_i$ for all $1\le l < i \le t$. Construct a sequence $(e_1,e_2,\ldots, e_t)$ recursively by $e_1=a_{d_1}$ and $$e_i=e_{i-1}+\frac{(a_{d_i}-e_{i-1})d_i}{n}\cdot\frac{x^n-1}{x^{d_i}-1}$$ for $1<i\le t$. Note that for $1<i\le t$, $e_i$ is constructed using Construction (1) where $c=e_{i-1}$, $r=a_{d_i}$, and $j=d_i$. Thus it follows from the properties of Construction (1) that $e_i \equiv a_{d_i} \mod \Phi_d(x)$ for all $d\mid d_i$, and that $e_i \equiv e_{i-1} \mod \Phi_d(x)$ for all $d\mid n$ such that $d\nmid d_i$. Using this, we can prove by induction over $i$ that $e_i \equiv a_{d_l} \mod \Phi_{d_l}(x)$ for each $1\le l\le i \le t$. Hence $e_t \equiv a_d \mod \Phi_{d}(x)$ for each $d\mid n$. Taking the $n$-simplification of $e_t$, we get an $n$-simplified polynomial $a$ satisfying $a \equiv a_d \mod \Phi_d(x)$ for each $d\mid n$. We know there can only be one such polynomial by Theorem \ref{thmmodcount}, completing the proof.
\end{proof}

\begin{rem}
 Note that since the polynomial $a$ constructed in Lemma \ref{lembuildresidues} exists and is unique, we have a formula for it due to Theorem \ref{thmmodcount}.
\end{rem}

\begin{thm}\label{thmrepelement}
Let $n$ be a positive integer. Let $a \in \mathbb{K}[x]$. Then $G^n(a)$ is the unique $n$-simplified polynomial which is congruent to $a \mod \Phi_n(x)$ and is congruent to $0 \mod x^d-1$ for each proper divisor $d$ of $n$.
\end{thm}
\begin{proof}
By Lemma \ref{lembuildresidues}, there is a unique $n$-simplified polynomial $f\in \mathbb{K}[x]$ which is congruent to $0 \mod \Phi_d(x)$ for each proper divisor $d$ of $n$ and which is congruent to $a \mod \Phi_n(x)$. Given an integer $i$ with $0\le i <n$, by Theorem \ref{thmmodcount}, we have that 
$$[x^i]f=\frac{1}{n}\left(G^n_i(f)+\sum\limits_{d\mid n,d<n} G^d_i(0)\right)=\frac{1}{n}G^n_i(f).$$
Thus $f=G^n(f)$. Since $g\equiv 0 \mod x^d-1 \implies g \equiv 0 \mod \Phi_d(x)$ for $g\in \mathbb{K}[x]$, it remains only to show that $f\equiv 0 \mod x^d-1$ for each proper divisor $d\mid n$ (and we will get uniqueness for free). Let $d$ be a proper divisor of $n$ and $i \in \mathbb{Z}$ with $0 \le i <d$. Since $f=G^n(f)$, we have $[x^k]f = [x^k]G^n(f) = \frac{1}{n} G^n_k(f) = \frac{1}{n}\sum\limits_{s\geq 0}[x^s]f\cdot c_n(k-s)$ for every $0\leq k < n$, so that $$\si^i_d(f)=\sum\limits_{0\le j < \frac{n}{d}}[x^{jd+i}]f=\frac{1}{n}\sum\limits_{0\le j <\frac{n}{d}}\sum\limits_{s\ge 0}[x^s]f\cdot c_n(jd+i-s)=\frac{1}{n}\sum\limits_{s\ge 0}[x^s]f\sum\limits_{0\le j <\frac{n}{d}} c_n(jd+i-s).$$
Noting Lemma \ref{lemramanujanroots}, the second sum on the right, $$\sum\limits_{0\le j <\frac{n}{d}} c_n(jd+i-s)=\sum\limits_{0\le j <\frac{n}{d}} \sum\limits_{w\in W_n}w^{jd+i-s}=\sum\limits_{w\in W_n}\sum\limits_{0\le j <\frac{n}{d}} w^{jd+i-s}=0.$$
Thus $\si^i_d(f)=0$, and as a consequence, $f \equiv 0 \mod x^d-1$.
\end{proof}

\begin{rem}
Let $n$ be a positive integer and suppose $\mathbb{K}$ is a field. Recalling that $\prod\limits_{k\mid d}\Phi_k(x)=x^d-1$ and that all $\Phi_k(x)$ are pairwise coprime, we see that for a polynomial in $\mathbb{K}[x]$ to be congruent to $0 \mod \Phi_d(x)$ for all proper divisors $d$ of $n$ is the same as it being congruent to $0 \mod x^d-1$ for all proper divisors $d$ of $n$. Thus when $\mathbb{K}$ is restricted to being a field, we have an alternative path to showing that $f \equiv 0 \mod x^d-1$ for proper divisors $d$ of $n$ (once we have already shown that $f \equiv 0 \mod \Phi_d(x)$).
\end{rem}

\begin{cor}\label{corcompleteinvariant}
Let $n$ be a positive integer, and $\mathbb{K}$ be a commutative ring
which is torsion-free as a $\mathbb{Z}$-module. Then, the map sending
every $a \in \mathbb{K}[x]$ to $\sum\limits_{0\leq i < n} G_i^n(a)
x^i$ is a complete invariant on $\mathbb{K}[x]$ modulo $\Phi_n(x)$.
\end{cor}
\begin{proof}
By Theorem \ref{thmrepelement}, this is true for $\mathbb{Q}[x] \mod \Phi_n(x)$ (and more generally,\\ $\mathbb{K}[x] \mod \Phi_n(x)$). Given $a,b \in \mathbb{Z}[x]$ with $G^n(a)=G^n(b)$, since $G^n$ is a complete invariant in $\mathbb{Q}[x] \mod \Phi_n(x)$, there exists some nonzero $j \in \mathbb{Z}$ such that $ja\equiv jb$ in $\mathbb{Z}[x] \mod \Phi_n(x)$. But $\mathbb{Z}[x] \mod \Phi_n(x)$ is free (since $\Phi_n(x)$ is monic), and hence is torsion-free. Thus $a\equiv b$ in $\mathbb{Z}[x] \mod \Phi_n(x)$ and $G^n$ is a complete invariant on $\mathbb{Z}[x] \mod \Phi_n(x)$ (rather than just an invariant).
\end{proof}

\begin{rem}
Corollary \ref{corcompleteinvariant} easily extends to become the following. Let $n$ be a positive integer. The map $\sum\limits_{0\le i <n}G^n_i(a)x^i$ is a complete invariant on $\mathbb{R}[x]$ where $R$ is a commutative ring torsion-free as a $\mathbb{Z}$-module.
\end{rem}

\section{Some Modular Enumeration Results}\label{Secap}
In this section, we demonstrate several applications of Theorem \ref{thmmodcount}. For some commonly studied generating functions $f$, formulas for residues modulo $\Phi_n(x)$ are already known while good formulas for $n$-simplifications are not \cite{Olive65,Sagan92,Desarmenien94,Lascoux94,Desarmenien83,Desarmenien92,Descouens06}. For such polynomials, one can apply Theorem \ref{thmmodcount} in order to get previously unknown modular enumeration results. 

For the sake of brevity, we only go into detail for one such generating function, the $q$-multinomial coefficient ${ j \brack k_1, k_2, \ldots, k_l }$ (Subsection \ref{Subbin}). Recently, Hartke and Radcliffe found a recursive formula for $S^i_n\left({ j \brack k_1, k_2, \ldots, k_l }\right)$ in the case of $n \mid j$ (Theorem 25 of \cite{Hartke13}). For the same case, we are able to use our approach to find a non-recursive formula (Remark \ref{remsimplemulti}). They pose finding a formula for the case of $n \nmid j$ where $n$ is not prime as an open problem. We resolve this problem with Theorem \ref{thmmulti}. Our formula reduces the problem to cases where $j<n$.

In Subsection \ref{Subsum}, we consider an open problem on subset sums posed in \cite{Kitchloo93} and we use Theorem \ref{thmmodcount} to reduce it to small cases.

\subsection{$q$-multinomial coefficients}\label{Subbin}

In this subsection, we consider the $q$-multinomial coefficient ${ j \brack k_1, k_2, \ldots, k_l }$, which satisfies $${ j \brack k_1, k_2, \ldots, k_l }=\frac{[j]!}{[k_1]!\cdots [k_l]!},$$ where $[s]!$ is defined as $1(1+q)(1+q+q^2)\cdots (1+q+\cdots + q^{s-1})$ and where we require $k_1+k_2+\cdots +k_l=j$. If $k_1+k_2+\cdots+k_l \neq j$, we define ${ j \brack k_1, k_2, \ldots, k_l }$ to be zero. Combinatorially, ${ j \brack k_1, k_2, \ldots, k_l }$ is the generating function keeping track of the distribution of inversion numbers of permutations of the multiset $\{1^{k_1},2^{k_2},\ldots, l^{k_l}\}$. Following standard conventions, we consider $q$-multinomial coefficients to be polynomials of $q$ rather than $x$.

An elementary proof of the following lemma can be found in \cite{Sagan92}.
\begin{lem}\label{lemmultiphi}
Let $j, k_1, k_2, \ldots, k_l \in \mathbb{Z}$. Let $j=j_1n+j_0$ and $k_i=k_{i,1}n+k_{i,0}$ where $0\leq n_0, k_{i,0} < n$ for $1 \leq i \leq l$. Then,
$${ j \brack k_1, k_2, \ldots, k_l } \equiv {j_1 \choose k_{1,1},k_{2,1},\ldots,k_{l,1}}{ j_0 \brack k_{1,0},k_{2,0},\ldots,k_{l,0}} \mod \Phi_n(q)$$
\end{lem}

We are now in a position to find an over-arching formula for the coefficients of the $n$-simplification of a $q$-multinomial coefficient.
\begin{thm}\label{thmmulti}
Let $n, j, k_1, k_2, \ldots, k_l \in \mathbb{Z}$. Let $r_d(s)$ be the residue of $s$ modulo $d$ for all $s$ and $d$ in $\mathbb{Z}$. Let $R_d=r_d(k_1)+r_d(k_2)+\cdots r_d(k_l)$. Then,
$$S^i_n\left( {j \brack k_1,k_2,\ldots,k_l} \right)=\frac{1}{n} \sum\limits_{\substack{d|n, \\ R_d<d}} {\lfloor \frac{j}{d} \rfloor \choose \lfloor \frac{k_1}{d} \rfloor,\lfloor \frac{k_2}{d}\rfloor,\ldots,\lfloor \frac{k_l}{d}\rfloor} \sum\limits_{0 \leq s <d} \si^s_d \left( {r_d(j) \brack r_d(k_1),r_d(k_2),\ldots,r_d(k_l)}  \right)c_d(i-s).$$

\end{thm}
\begin{proof}
By Corollary \ref{cormodcount} and Lemma \ref{lemmultiphi}, 
$$S^i_n\left( {j \brack k_1,k_2,\ldots,k_l} \right)=\frac{1}{n} \sum\limits_{d|n}\sum\limits_{0 \leq s <d} \si^s_d \left( {\lfloor \frac{j}{d} \rfloor \choose \lfloor \frac{k_1}{d} \rfloor,\lfloor \frac{k_2}{d}\rfloor,\ldots,\lfloor \frac{k_l}{d}\rfloor}{r_d(j) \brack r_d(k_1),r_d(k_2),\ldots,r_d(k_l)}  \right)c_d(i-s).$$
Pulling out the $ {\lfloor \frac{j}{d} \rfloor \choose \lfloor \frac{k_1}{d} \rfloor,\lfloor \frac{k_2}{d}\rfloor,\ldots,\lfloor \frac{k_l}{d}\rfloor}$ term yields the formula
$$S^i_n\left( {j \brack k_1,k_2,\ldots,k_l} \right)=\frac{1}{n} \sum\limits_{d|n} {\lfloor \frac{j}{d} \rfloor \choose \lfloor \frac{k_1}{d} \rfloor,\lfloor \frac{k_2}{d}\rfloor,\ldots,\lfloor \frac{k_l}{d}\rfloor} \sum\limits_{0 \leq s <d} \si^s_d \left( {r_d(j) \brack r_d(k_1),r_d(k_2),\ldots,r_d(k_l)}  \right)c_d(i-s).$$
Observing when the $q$-multinomial coefficient is zero yields the desired formula.
\end{proof}

\begin{rem}\label{remsimplemulti}
Theorem \ref{thmmulti} simplifies in the case of $n \mid j$. Indeed, for each $d|n$, ${r_d(j) \brack r_d(k_1),r_d(k_2),\ldots,r_d(k_l)}$ is $0$ if $d\nmid (k_1, k_2, \ldots, k_l)$ and is $1$ otherwise. Thus we have
$$S^i_n\left( {j \brack k_1,k_2,\ldots,k_l} \right)=\frac{1}{n} \sum\limits_{d \mid (k_1, k_2, \ldots, k_l, n)} { \frac{j}{d} \choose\frac{k_1}{d}, \frac{k_2}{d},\ldots, \frac{k_l}{d} } c_d(i).$$
\end{rem}

\subsection{Subset-sums}\label{Subsum}
In this subsection, we consider the number of subsets of $\{1,2,\ldots,j\}$ whose sum is congruent to $i$ modulo $n$. In doing so, we make progress on an open problem posed in \cite{Kitchloo93}.

Let $N(j)$ be the polynomial $\sum\limits_{S\subseteq \{1,2,\ldots,j\}} x^{\sum S}$. (Here, $\sum S$ denotes the sum of all elements of $S$.)

\begin{lem}\label{lemsubset-sum}
Let $d$ and $j$ be integers such that $0 < d \leq j$, and let $r$ be the remainder of $j$ modulo $d$. We have that $N(j) \equiv 0 \mod \Phi_d(x)$ if $d$ is even, and $N(j) \equiv 2^{\lfloor j/d \rfloor}N(r) \mod \Phi_d(x)$ otherwise. 
\end{lem}
\begin{proof}
Clearly, $N(j)=N(d)^{\lfloor j/d \rfloor}N(r)$. Since $\lfloor j/d \rfloor >0$, it remains to show only that $N(d) \equiv 0 \mod \Phi_d(x)$ if $d$ is even and $N(d) \equiv 2 \mod \Phi_d(x)$ if $d$ is odd. Each subset $S$ of $\{1,2,\ldots,d\}$ corresponds with a binary word $w$ of length $d$ where $w_i=1$ if $i\in S$ and $w_i=0$ otherwise. Let $w$ be an arbitrary binary word of length $d$. 

Suppose $w$ consists neither of just ones nor of just zeros. Let $k$ be the number of ones in $w$. Then applying successive cyclic shifts to $w$, that is repeatedly killing the final letter and reinserting it as the first letter, increases the corresponding subset sum by $k \mod d$ with each shift. The generating function for the resulting subset sums is thus periodic on $d$ with period $k$. By Lemma \ref{lemperiodzero}, this generating function is congruent to $0 \mod \Phi_d(x)$. 

If $d$ is odd, the remaining two cases for $w$, where $w$ has just ones or has just zeros, both correspond with subsets whose subset sums are congruent to $0 \mod d$. Hence $N(d) \equiv 2 \mod \Phi_d(x)$ in this case. If $d$ is even, the empty subset has subset sum congruent to $0 \mod d$ and the subset $\{1,2,\ldots,d\}$ has subset sum congruent to $\frac{d}{2} \mod d$ (because $1+2+\cdots+d=\frac{d(d+1)}{2}$). Since $1+x^{d/2}\equiv 0 \mod \Phi_d(x)$, it follows that $N(d) \equiv 0 \mod \Phi_d(x)$. 
\end{proof}

\begin{thm}\label{thmsubset-sum}
Let $n$ and $j$ be integers such that $0<n\le j$. Let $r_d$ be the remainder of $j$ modulo $d$ for each odd $d\mid n$. Then
$$\si^i_n(N(j))= \frac{1}{n}\sum_{\substack{d\mid n,\\ d \text{ odd}}}2^{\lfloor j/d \rfloor}\sum_{0\le s <d} {\si^s_d(N(r_d))}c_d(i-s).$$
\end{thm}
\begin{proof}
Plugging Lemma \ref{lemsubset-sum} into Corollary \ref{cormodcount} and pulling out the exponent term, we get the desired result.
\end{proof}



\begin{rem}
Each of \cite{Stanley72}, \cite{Kitchloo93}, and \cite{Li12} previously enumerated $\si^i_n(N(j))$ either for the case of $n=j$ or for the case of $j=n-1$. In \cite{Stanley72}, the problem arose in the context of single-error correction codes. Kitchloo and Pachter, in \cite{Kitchloo93}, were able to enumerate $\si^0_n(N(j))$ in the cases where $n\mid j$, and posed the remaining cases as an open problem. One consequence of Theorem \ref{thmsubset-sum} is a nice extension of \cite{Kitchloo93}'s formula; in fact, when $n\le j$ and for each odd $d\mid n$, $d\mid j$, we have $$\si^i_n(N(j))= \frac{1}{n}\sum\limits_{\substack{d\mid n,\\ d \text{ odd}}}2^{\lfloor j/d \rfloor}c_d(i).$$ As is the case for this formula, the $d$ odd requirement is behind much of the computational power of Theorem \ref{thmsubset-sum}. For example, as a consequence of Theorem \ref{thmsubset-sum}, the number of subsets of $\{1, \ldots, 22\}$ which have subset sum congruent to $5 \mod 12$ is 
$$\frac{1}{12}\sum_{\substack{d\mid 12,\\ d \text{ odd}}}2^{\lfloor 22/d \rfloor}\sum_{0\le s <d} {\si^s_d(N(22-d\lfloor \frac{22}{d}\rfloor))}c_d(5-s)$$
$$= \frac{1}{12}(2^{22}+2^7(c_3(5)+c_3(4)))= \frac{1}{12}(2^{22}+2^7(-2))=349504,$$
a conclusion which intuitively one should not be able to quickly reach by hand.
\end{rem}

\section{Major Index of Dyck Words}\label{Seccat}
In this section, ``\emph{word}'' means ``finite binary word''. We denote the $i$-th letter of a word $w$ by $w_i$.

\begin{defn}
A word is \emph{flat} if it contains at least as many zeros as ones.
\end{defn}
For example, $11000$ and $001110$ are flat. On the other hand, $00111$ is not because it contains more ones than zeros. Note that if one thinks of a binary word as representing a path, with each $0$ corresponding with a step to the right and each $1$ corresponding with a vertical step, then flat words correspond with paths that are flatter than they are tall.

\begin{defn} A word is a \emph{Dyck word} if its first $k$ letters form a flat word for all $k$. Otherwise, the word is \emph{non-Dyck}. 
\end{defn}
For example, $001101000$ is a Dyck word. On the other hand, $001110000$ is not because its first $5$ letters do not form a flat word. 

\begin{defn}
The \emph{major index} of a word $w$, denoted $m(w)$, is $\sum\limits_{w_i=1,w_{i+1}=0}i$.
\end{defn}
For example, the major index of $0011000101001$ is $4+8+10=22$ because the occurrences of a one followed by a zero are in positions $4$, $8$, and $10$. Note that the final letter is not considered to be followed by the first letter.

This notion of major index, when applied to Dyck words with $j$ ones and $j$ zeros, corresponds with major index of $j \times j$ Catalan paths, as defined in \cite{Haglund08}. The generating function for this statistic is the $q$-Catalan number $C_n(q)$. The $q$-Catalan number, one of the most natural $q$-analogues of the Catalan numbers, satisfies $C_n(q)=\frac{1-q}{1-q^{n+1}}{2n \brack n}$. Although the coefficients of $C_n(q)$ have no known formula, in this section we are able to find a formula for $\mathbf{S}^i_n(C_j(q))$ for given $j, i, n$. In some cases the formula is complete, while in others it reduces the problem to cases where $j<n$; since the number of such cases is finite for a fixed $n$, our formula can be used to find a non-recursive one for any fixed $n$. In finding our results, we first enumerate remainders of $q$-Catalan numbers modulo cyclotomic polynomials. We prove our enumeration in two separate ways, the first in the context of major index of Catalan paths, and the second using generating functions. In the prior, we also introduce a previously unknown cyclic group operation which interacts interestingly with major index of binary words. 

\begin{defn}
Let $w$ be a word. Then $\gamma(w)$ is $w$ with its final letter killed and then appended to the beginning. We call the words which can be reached from $w$ by repeated applications of $\gamma$ the \emph{cyclic shifts} of $w$. Note that $\gamma$ is invertible.
\end{defn}
For example, the cyclic shifts of $0010110$ are $0010110$, $0001011$, $1000101$, $1100010$, $0110001$, $1011000$, and $0101100$.

\begin{defn}
Let $w$ be a flat non-Dyck word. Then $\delta(w)$ is defined as follows. Let $w'=\gamma(w)$. If $w'$ is non-Dyck, then $\delta(w)=w'$ (\emph{shifting case 1}). Otherwise, find the smallest positive $k$ such that the first $k$ letters of $w'$ contain the same number of ones and zeros. Because $w$ is flat and non-Dyck, such a $k$ exists. Since $w'$ is a Dyck word, its first letter is zero and its $k$-th letter is one. Swapping these two letters, we get a flat non-Dyck word $\delta(w)$ (\emph{shifting case 2}). 
\end{defn}

Note that the domain and range of $\delta$ is flat non-Dyck words. In addition, $\delta$ preserves the number of ones and the number of zeros in a word. For example, $\delta (11000)=01100$ because $\gamma (11000)=01100$ is non-Dyck (an example of shifting case 1). Now let's compute $\delta (01100)$. Since $\gamma(01100)=00110$ which is Dyck, we are in shifting case 2. Swapping the first and fourth letters, we reach $\delta(01100)=10100$.

\begin{lem}
The map $\delta$ is invertible on flat non-Dyck words. 
\end{lem}
\begin{proof}
To prove the invertibility of a map from a finite set to itself, we need only check the map's surjectivity. Now let us show that $\delta$ is surjective. 

If a flat non-Dyck word $w$ begins with zero, it is easy to see that $w=\delta(\gamma^{-1}(w))$ because $\gamma^{-1}(w)$ is non-Dyck. If a flat non-Dyck word $w$ begins with one, then there are two cases:

\emph{First case:} There is some $k$ such that the first $k$ letters of $w$ contain at least two more ones than they do zeros. In this case, $\delta(\gamma^{-1}(w))=w$ because $\gamma^{-1}(w)$ is non-Dyck.

\emph{Second case:} For all $k$, the first $k$ letters of $w$ contain at most one more one than they do zeros. Let $k$ be the largest $k$ such that the first $k-1$ letters of $w$ contain more ones than zeros and $k\leq n$. (Such a $k$ exists because $w$ is non-Dyck.) Note that $w_k=0$ because $w$ is flat. Swapping the first and $k$-th letter of $w$ yields a Dyck word $u$. Note that the smallest positive $j$ such that the first $j$ letters of $u$ contain the same number of ones and zeros satisfies $j=k$. Furthermore, $\gamma^{-1}(u)$ is non-Dyck because such a $j$ exists and $u_1=0$. It follows that $\delta(\gamma^{-1}(u))=w$.
\end{proof}

\begin{defn}
The \emph{descent count} of a word $w$ of length $n$, denoted $d(w)$, is $$(1 \text{ if }w_n=1,w_1=0)+\sum\limits_{w_i=1,w_{i+1}=0}1.$$ 
\end{defn}
For example, the descent count of $0011000101001$ is $4$ because there are three occurrences of a one followed by a zero (in positions $4$, $8$, and $10$), and because $w_n=1$ and $w_1=0$.

Note that $d(\gamma(w))=d(w)$ and $m(\gamma(w))-m(w)\equiv d(w) \mod n$ for all words $w$ of length $n$.\footnote{This is because $\gamma$ simply increases the position of each occurrence of a $1$ followed directly by a $0$ in $w$ by one modulo $n$.} We will use this implicitly from here out.

\begin{prop}\label{propcycle}
Let $w$ be a flat non-Dyck word of length $n$ containing at least two ones (and thus at least two zeros). Then $$m(\delta(w))-m(w) \equiv m(\delta(\delta(w)))-m(\delta(w)) \not\equiv 0 \mod n.$$
\end{prop}
\begin{proof}
Let $k$ be the smallest positive $k$ such that the first $k$ letters of $\gamma(w)$ contain the same number of zeros as ones. There are six cases:
\begin{enumerate}
\item $w$ is in shifting case 1. Then $\delta(w)=\gamma(w)$. Thus $m(\delta(w))-m(w) \equiv d(w) \mod n$, and $d(\delta(w))=d(w)$. Furthermore, $\delta(w)$ is in one of cases (1), (2), (3), and (4). Indeed, $\delta(w)$ cannot fall into cases (5) and (6) because for words $y$ in those cases, $\gamma^{-1}(y)$ is a Dyck word.
\item $w$ is in shifting case 2, $k\neq 2$, $k\neq n$, and $\gamma(w)_n=1$. These restrictions imply that $\gamma(w)_1=\gamma(w)_2=0$, $\gamma(w)_{k-1}=1$, and $\gamma(w)_{k+1}=0$. As a consequence, $m(\delta(w))-m(w)\equiv d(w) \mod n$ and $d(\delta(w))=d(w)$. Furthermore, $\delta(w)$ must be in case (1).
\item $w$ is in shifting case 2, $k\neq 2$, $k\neq n$, and $\gamma(w)_n=0$. These restrictions imply that $\gamma(w)_1=\gamma(w)_2=0$, $\gamma(w)_{k-1}=1$, and $\gamma(w)_{k+1}=0$. As a consequence, $m(\delta(w))-m(w)\equiv d(w) \mod n$ and $d(\delta(w))=d(w)+1$. Furthermore, $\delta(w)$ falls in one of cases (5) and (6).
\item $w$ is in shifting case 2 and $k=n$. It follows that $\gamma(w)_1=\gamma(w)_2=0$ and $\gamma(w)_{n-1}=\gamma(w)_{n}=1$. As a consequence, $m(\delta(w))-m(w)\equiv d(w) \mod n$ and $d(\delta(w))=d(w)+1$. Furthermore, $\delta(w)$ can fall into only case (5).
\item $w$ is in shifting case 2, $k=2$, and $\gamma(w)_n=1$. Note $\gamma(w)_{k+1}=0$. Thus $m(\delta(w))-m(w)\equiv d(w)-1 \mod n$ and $d(\delta(w))=d(w)-1$. Furthermore, $\delta(w)$ is in case (1).
\item $w$ is in shifting case 2,  $k=2$, and $\gamma(w)_n=0$. Note $\gamma(w)_{k+1}=0$. Thus $m(\delta(w))-m(w) \equiv d(w)-1 \mod n$ and $d(\delta(w))=d(w)$. Furthermore, $\delta(w)$ is in one of cases (5) and (6).
\end{enumerate}
The above cases imply that $m(\delta(w))-m(w) \equiv m(\delta(\delta(w)))-m(\delta(w)) \mod n$. For $m(\delta(w))-m(w) \equiv 0 \mod n$ to be true, $w$ would need to be in one of cases (5) and (6) with $d(w)<2$. This cannot happen when $w$ contains at least two ones and at least two zeros. 
\end{proof}


\begin{defn}
Let $d$ be an integer greater than one. We say a Dyck word $w$ of length $kd$ is \emph{$d$-rigid} if the following is true:
\begin{itemize}
\item The subword $w_{jd+1}\cdots w_{jd+d}$ for each $0\le j<k$ either is $d$ ones, is $d$ zeros, or contains precisely $d-1$ zeros (in this case, the subword is called a \emph{rigid-type-2 interval}), or contains precisely $d-1$ ones (in this case, the subword is called a \emph{rigid-type-3 interval}). 
\item There are the same number of zeros as ones preceding any rigid-type-2 interval, and the same number of zeros as ones in the letters from the beginning of $w$ to the end of any rigid-type-3 interval.
\end{itemize}
\end{defn}
For example, $000111001000111101$ is $3$-rigid. So is $000111010000111110$, since the order of the letters in rigid-type-2 an rigid-type-3 intervals are not restricted in the definition.

\begin{defn}
Let $d$ be an integer greater than one. A \emph{$d$-straightened Dyck word} $w$ is a $d$-rigid Dyck word such that all of its rigid-type-2 intervals end in one and all of its rigid-type-3 intervals begin with zero.
\end{defn}
For example, $000111001000111011$ is $3$-straightened. Note that $000111010000111110$ is not $3$-straightened, even though it is $3$-rigid.

\begin{defn}
Let $T_d(j,k)$ denote the number of $d$-straightened Dyck words containing $dj$ zeros and $dk$ ones if $j \ge k$. Otherwise, let $T_d(j,k)$ denote the number of $d$-straightened Dyck words containing $dk-1$ zeros and $dj+1$ ones.
\end{defn}

\begin{lem}\label{lemgetT}
Let $d,j,k$ be non-negative integers such that $d>1$. Then,
$$T_d(j,k)={{j+k}\choose{k}}.$$
\end{lem}
\begin{proof}

If either $j$ or $k$ equals zero, it is easy to see that $T_d(j,k)=1$. We now consider the remaining cases, showing that $T_d(j,k)=T_d(j-1,k)+T_d(j,k-1)$ in each case. If $j>k$, then every $d$-straightened Dyck word containing $dj$ zeros and $dk$ ones ends either with $d$ ones or with $d$ zeros. If $j=k$, then every $d$-straightened Dyck word containing $dj$ zeros and $dk$ ones ends either with $d$ ones or with a zero followed by $d-1$ ones. If $j=k-1$, then every $d$-straightened Dyck word containing $dk-1$ zeros and $dj+1=d(k-1)+1$ ones ends either with $d$ ones or with $d-1$ zeros followed by a one. If $j<k-1$, then every $d$-straightened Dyck word containing $dk-1$ zeros and $dj+1$ ones ends either with $d$ ones or with $d$ zeros. Noting the definition of $T_d(j,k)$, it follows that in each of these cases, $T_d(j,k)=T_d(j-1,k)+T_d(j,k-1)$; thus we have recursively demonstrated the lemma.
\end{proof}








For the rest of the section, let $X(a,b)$ be the generating function for major index of Dyck words containing exactly $a$ ones and $b$ zeros. Let $Y_n(a,b)$ be the generating function for major index of $n$-rigid Dyck words containing exactly $a$ ones and $b$ zeros.

\begin{lem}\label{lemfilterwords1}
Fix non-negative integers $a,b,k,n$ such that $a+b=kn$ and $n>1$. Then, 
$$X(a,b) \equiv Y_n(a,b) \mod \Phi_n(q) \text{ in } \mathbb{Z}[q].$$
\end{lem}
\begin{proof}
Recall that the generating function for major index of the cyclic shifts of a word $w$ of length $n$ containing at least one $1$ and one $0$ is periodic on $n$ with period $d(w)$.\footnote{This is because $m(\gamma(w))-m(w)\equiv d(w) \mod n$, as we previously noted.} By Lemma \ref{lemperiodzero}, this implies the generating function for major index of words of length $n$ containing $t$ ones and $u$ zeros where $t,u \neq 0$ is congruent to $0 \mod \Phi_n(q)$ (Observation (1)). By Proposition \ref{propcycle}, the generating function for major index of the images under iteration of $\delta$ of a flat non-Dyck word of length $n$ containing at least two ones is periodic on $n$, and thus also congruent to $0 \mod \Phi_n(q)$. Noting Observation (1), this implies that the generating function for major index of Dyck words of length $n$ containing $t$ ones and $u$ zeros where $t\ge 2$ (and thus $u\ge 2$) is congruent to $0 \mod \Phi_n(q)$ (Observation (2)).

Let $w$ be a Dyck word of length $kn$ containing $a$ ones and $b$ zeros that is not $n$-rigid. Let $l$ be the largest $l$ such that the first $ln$ letters of $w$ form an $n$-rigid word. Note that the first $ln$ letters of $w$ either have the same number of ones as zeros (first case), or have at least $n-2$ more zeros than ones (second case). 
By Observation (2), the generating function for major index of the words $w$ falling into the first case is congruent to $0 \mod \Phi_n(q)$ (because we can look at possibilities for the subword $w_{ln+1} \cdots w_{(l+1)n}$; we are implicitly using that $m(w)\equiv m(w_1\cdots w_{ln})+m(w_{ln+1}\cdots w_{(l+1)n})+m(w_{(l+1)n+1}\cdots w_{kn}) \mod n$). Similarly, by Observation (1), the generating function for major index of the words $w$ that fall into the second case is congruent to $0 \mod \Phi_n(q)$. Hence $X(a,b) \equiv Y_n(a,b)\mod \Phi_n(q)$.
\end{proof}

\begin{lem}\label{lemfilterwords2}
Fix non-negative integers $a,b,k,n$ such that $a+b=kn$ and $n>1$. Let $z$ be the number of $n$-straightened Dyck words containing exactly $a$ ones and $b$ zeros. Let $Z=z$ if $a,b \equiv 0 \mod n$; let $Z=-qz$ if $a \equiv 1 \mod n$ and $b \equiv -1 \mod n$; and let $Z=0$ otherwise. Then $Y_n(a,b) \equiv Z \mod \Phi_n(q)$. 
\end{lem}
\begin{proof}
It is easy to see that the generating function for major index of Dyck words of length $n$ which contain exactly one $1$ is congruent to $-q \mod \Phi_n(q)$; and similarly, the generating function for words of length $n$ which contain exactly one zero and do not end with zero is congruent to $-q^{n-1} \mod \Phi_n(q)$ (Observation (3)).  

If $a,b \equiv 0 \mod n$, then the number of rigid-type-2 intervals equals the number of rigid-type-3 intervals in each word considered by $Y_n(a,b)$. If  $a \equiv 1 \mod n$ and $b \equiv -1 \mod n$, then there is exactly one more rigid-type-2 interval than there are rigid-type-3 intervals in each word considered by $Y_n(a,b)$. In the remaining cases, there are no words considered by $Y_n(a,b)$. Note that while maintaining a $d$-rigid Dyck word's Dyckness, one can mutate any rigid-type-2 interval (resp. rigid-type-3 interval) by moving around the single $1$ (resp. the single $0$) within the interval, as long as the $1$ never takes the first position (resp. the $0$ never takes the final position) in the interval. As a consequence, we can apply Observation (3) to conclude that $Y_n(a,b) \equiv Z \mod \Phi_n(q)$.
\end{proof}

\begin{lem}\label{lemcatres}
Fix non-negative integers $j$ and $n$ such that $n>1$. Let $g$ be the remainder of $j$ modulo $n$. If $2g \geq n$ and $g\neq n-1$, then $X(j,j)\equiv 0 \mod \Phi_n(q)$. If $2g<n$, then $X(j,j)\equiv \dbinom{\frac{2j-2g}{n}}{\frac{j-g}{n}}X(g,g) \mod \Phi_n(q)$. If $g=n-1$, then $X(j,j)\equiv -q \dbinom{\frac{2j-n+2}{n}} {\frac{j+1}{n}} \mod \Phi_n(q)$.
\end{lem}
\begin{proof}
Let $X'(j,j)$ be the generating function for major index of Dyck words containing exactly $j$ ones and $j$ zeros such that the first $\lfloor \frac{2j}{n} \rfloor n$ letters form a $d$-rigid word. By Lemma \ref{lemfilterwords1}, $X(j,j)\equiv X'(j,j) \mod \Phi_n(q)$. Observe that when $2g \geq n$ and $g\neq n-1$, $X'(j,j)=0$. 

Suppose $2g<n$. For this case, observe that $X'(j,j)=Y_n(j-g,j-g)X(g,g)$. By Lemmas \ref{lemfilterwords2} and \ref{lemgetT}, this implies that $X'(j,j)\equiv \dbinom{\frac{2j-2g}{n}}{\frac{j-g}{n}}X(g,g) \mod \Phi_n(q)$.

Suppose $g=n-1$. For this case, each word considered by $X'(j,j)$ ends with $n-2$ ones. Therefore, $X'(j,j)=Y_n(j,j-n+2)$. By Lemmas \ref{lemfilterwords2} and \ref{lemgetT}, $Y_n(j,j-n+2)\equiv  -q\dbinom {\frac{2j-n+2}{n}} {\frac{j+1}{n}} \mod \Phi_n(q)$.
\end{proof}

We now provide an alternative proof of Lemma \ref{lemcatres} using the known formula for $C_n(q)$. This proof, although shorter, is less combinatorially interesting than the previous one.

\begin{proof}
It is known that $C_j(q)=\frac{1-q}{1-q^{j+1}}{2j \brack j}$. Let $g$ be the remainder of $j$ modulo $n$ and $h$ be the remainder of $2j$ modulo $n$. Recall that by the $q$-Lucas theorem (proven in \cite{Olive65}), $${2j \brack j}\equiv {\lfloor 2j/n \rfloor \choose \lfloor j/n \rfloor}{h \brack g}\mod{\Phi_n(q)}.$$

When $g>h$, ${h \brack g}=0$. Thus, when $g>h$ and $1-q^{j+1}$ is invertible modulo $\Phi_n(q)$, it follows that $C_j(q) \equiv 0 \mod \Phi_n(q)$. Since $\Phi_n(q)$ is irreducible, we conclude that when $2g \geq n$ and $g \neq n-1$, $C_j(q)\equiv 0 \mod \Phi_n(q)$.

Consider instead the case of $2g <n$. Plugging $C_j(q)$ into the $q$-Lucas theorem, and noting that $q^j \equiv q^g \mod \Phi_n(q)$, we get
$$\frac{1-q}{1-q^{j+1}} {2j \brack j}\equiv {\lfloor 2j/n \rfloor \choose \lfloor j/n \rfloor}\frac{1-q}{1-q^{g+1}}{h \brack g}={\lfloor 2j/n \rfloor \choose \lfloor j/n \rfloor}C_g(q) \mod{ \Phi_n(q)}.$$

In the remaining case, $g=n-1$. Note the identity $C_j(q)=\frac{1-q}{1-q^{2j+1}}{2j+1 \brack j}.$ Applying the $q$-Lucas Theorem, and noting that $q^{2j+1} \equiv q^{-1} \mod \Phi_n(q)$, we get that
$$\frac{1-q}{1-q^{2j+1}} {2j+1 \brack j}\equiv {\lfloor (2j+1)/n \rfloor \choose \lfloor j/n \rfloor}\frac{1-q}{1-q^{-1}}{g \brack g}\mod{\Phi_n(q)}.$$
This, in turn, simplifies to $-q{\lfloor (2j+1)/n \rfloor \choose \lfloor j/n \rfloor}.$
\end{proof}

We can now accomplish the goal of this section. We define $^jM^i_n$ to denote the number of Dyck words $w$ containing exactly $j$ ones and $j$ zeros\footnote{Note that we consider the empty word to be a word.} which satisfy $m(w) \equiv i \mod n$. 

\begin{thm}\label{thmcat}
Let $j,i,n$ be fixed non-negative integers. Let $r_d$ be the remainder of $j$ modulo $d$ for $d\mid n$. Let $C_j$ be the $j$-th Catalan number. Then,
$$^jM^i_n=\frac{1}{n}\left(C_j+\sum_{\substack{d\mid n, \\ \lfloor \frac{2j}{d} \rfloor =  \ 2 \lfloor \frac{j}{d} \rfloor, \\ d \neq 1}} {{\lfloor \frac{2j}{d} \rfloor} \choose {\lfloor \frac{j}{d} \rfloor}}\sum_{\substack{0  \le s  <d}}{^{r_d}M^s_d} c_d(i-s)-\sum_{\substack{d\mid n, \\ j\equiv  -1  \mod d, \\ d\neq 1}}{{\lfloor \frac{2j}{d} \rfloor}\choose {\lfloor \frac{j}{d}\rfloor}} c_d(i-1) \right).$$
\end{thm}
\begin{proof}
Noting that $X(j,j) \equiv C_j \mod \Phi_1(q)$, we can apply Lemma \ref{lemcatres} and Theorem \ref{thmmodcount} to get
$$^jM^i_n=\frac{1}{n}\left(C_j+\sum_{\substack{d\mid n, \\ \lfloor \frac{2j}{d} \rfloor =  2 \lfloor \frac{j}{d} \rfloor, \\ d \neq 1}}\sum_{\substack{0 \le s <d}}{{\lfloor \frac{2j}{d} \rfloor} \choose {\lfloor \frac{j}{d} \rfloor}}{^{r_d}M^s_d} c_d(i-s)-\sum_{\substack{d\mid n, \\ j\equiv -1  \mod  d, \\ d\neq 1}}{{\lfloor \frac{2j}{d} \rfloor}\choose {\lceil \frac{j}{d}\rceil}} c_d(i+n-1) \right).$$
Rearranging slightly, noting that $\displaystyle {{\lfloor \frac{2j}{d} \rfloor}\choose {\lceil \frac{j}{d}\rceil}}={{\lfloor \frac{2j}{d} \rfloor}\choose {\lfloor \frac{j}{d}\rfloor}}$, and noting that $c_d(i+n-1)=c_d(i-1)$, we get the desired formula.
\end{proof}

\begin{rem}
When $n\mid j$, Theorem \ref{thmcat} simplifies to
$^jM^i_n=\frac{1}{n}\left(C_j+\sum_{\substack{d\mid n, \\ d \neq 1}} {{2j/d} \choose {j/d}}c_d(i)\right).$
When $n\mid j-1$, Theorem \ref{thmcat} simplifies to
$^jM^i_n=\frac{1}{n}\left(C_j-\sum_{\substack{d\mid n, \\ d\neq 1}}{{\lfloor \frac{2j}{d} \rfloor}\choose {\lfloor \frac{j}{d}\rfloor}} c_d(i-1) \right).$ Furthermore, to obtain a non-recursive formula for all cases of $^jM^i_n$ for any fixed $n$, one needs only compute $n^2$ base cases.\footnote{We need only to compute the coefficients of $\si_n(C_j(q))$ for $j<n$ because for $d\mid n$, $^jM^i_d$ can be expressed as the sum of $^jM^t_n$ over $t$ such that $0\le t<n$ and $t\equiv i \mod d$.}
\end{rem}

\section{Discussion and Future Work}\label{Conclusion}

Let $f \in \mathbb{Z}[x]$. Until now, the most versatile (and commonly used) tool for finding $\si_n(f)$ has been what is sometimes referred to as the roots of unity filter \cite{Yuan09} (really an application of the discrete Fourier transform), which uses that $\si^i_n(f)=\sum\limits_{w^n=1}w^{-i}f(w).$ Evaluating $f$ at roots of unity is roughly the same problem as finding coefficients for a polynomial congruent to $f$ modulo cyclotomic polynomials (step (1)). From step (1), however, the roots of unity filter leaves users with an expression which, in practice, only simplifies nicely in the case where $f(w)$ is an integer for all $w^n=1$.\footnote{In this simple case, the order of $w$ determines $f(w)$. Thus one can use Ramanujan sums to get rid of the roots of unity and obtain the formula that would be found by applying Theorem \ref{thmmodcount}. Alternatively, as shown in \cite{Desarmenien90}, one can use Lagrange interpolation to the same end. Cohen, in \cite{Cohen55}, appears to have been the first to have noted the formula for this case.} On the other hand, Theorem \ref{thmmodcount} takes one directly from step (1) to an elegant formula. In example applications from Sections \ref{Secap} and \ref{Seccat}, this has allowed for us to solve several problems which were previously unapproachable; hopefully this trend will continue in future work.

We conclude with two directions of future work which we have found in our research.
\begin{enumerate}
\item In every congruence class modulo $\Phi_n(x)$, there exists an $n$-simplified polynomial $a \in \mathbb{Z}[x]$ with minimum $\sum\limits_i |[x^i]a|$. What can one say about such $a$? Is there a greedy algorithm to find such an $a$? When is $a$ unique? 
\item We have recently proven the following two results about $\delta$ (as defined in Section \ref{Seccat}). 
\begin{prop}
Let $w$ be a flat non-Dyck word of size $n$. Let $w'$ be the word reached from $w$ by applying $\gamma$ as many times as possible without reaching a Dyck word (we will refer to this as $j$ times), and then applying $\delta$ (which will be in shifting-case 2). Let $w''$ be $w'$ except with its first letter replaced by a zero. Let $S(w)$ be the set containing the largest $i$ such that there are $k$ more zeros than ones in the final $n-i+1$ letters of $w''$ (for each $k$ where such an $i$ exists). For a given $m$, let $\beta^{m}(w)$ denote $\gamma^m(w'')$ except with the $l$-th letter replaced by a one, where $l$ is the smallest element of
$\{\text{remainder of } h+m\text{ modulo }n\ :\ h \in S(w)\}.$  Then, $\delta^m(w)=\beta^{m-j-1}(w)$.
\end{prop}
\begin{cor}
Let $w$ be a flat non-Dyck word of size $n$. Then, $\delta^n(w)=w.$
\end{cor}
There are many additional questions to ask about $\delta$. For example, how many equivalence classes are there of non-Dyck words of size $n$ modulo $\delta$?
\end{enumerate}

\vspace{4cm}

\begin{center}
\textbf{Acknowledgements}
\end{center}

I would like to thank the MIT PRIMES program for providing me with the resources to conduct this research project. I thank Darij Grinberg for many useful conversations throughout the project, as well as for helping me put my thoughts on paper. His help was invaluable. I also thank MIT Prof. Richard Stanley for useful conversations about direction of research.

\newpage
\bibliographystyle{plain}
\pagestyle{empty}\singlespace
\bibliography{modulostats}

\end{document}